\documentclass{amsart}
\usepackage{graphicx}
\usepackage{amscd}
\usepackage{amsmath}
\usepackage{amsfonts}
\usepackage{amssymb}
\newtheorem{theorem}{Theorem}
\theoremstyle{plain}

\newtheorem{corollary}{Corollary}

\newtheorem{lemma}{Lemma}

\newtheorem{remark}{Remark}

\numberwithin{equation}{section}

\begin{document}
\title[Coifman-Rochberg-Weiss commutator]{An abstract Coifman-Rochberg-Weiss commutator theorem}
\author{Joaquim Martin}
\address{Department of Mathematics\\
Universidad Autonoma de Barcelona}
\email{jmartin@mat.uab.es}
\urladdr{}
\author{Mario Milman}
\curraddr{Department of Mathematics\\
Florida Atlantic University\\
Boca Raton, Florida 33431}
\email{extrapol@bellsouth.net}
\urladdr{http://www.math.fau.edu/milman}
\thanks{2000 Mathematics Subject Classification Primary: 46E30, 26D10.}
\thanks{The first author is supported in part by MTM2007-60500 and by CURE 2005SGR 00556.}
\thanks{This paper is in final form and no version of it will be submitted for
publication elsewhere.}
\dedicatory{It is a special pleasure for us to dedicate this paper to you, our dear friend
Dan Waterman on the occasion of your 80th birthday. But that is not all.The
things we discuss here are intimately related to important work by another
dear friend, and so to you too Richard Rochberg, warmest greetings on the
occasion of your 65th birthday. We wish both of you many many more wonderful
and creative years.}

\begin{abstract}
We formulate and prove a version of the celebrated Coifman-Rochberg-Weiss
commutator theorem for the real method of interpolation
\end{abstract}\maketitle

\section{Introduction}

Commutator estimates play an important role in analysis (cf. \cite{se}). Our
starting point in this paper is the celebrated commutator theorem of
Coifman-Rochberg-Weiss \cite{crw}. Let $K$ be a Calder\'{o}n-Zygmund operator,
and let $b\in BMO(R^{n}).$ Denote by $M_{b}$ the operator ``multiplication by
$b$'', then (cf. \cite{crw})%
\begin{equation}
\left\|  \lbrack K,M_{b}]f\right\|  _{p}\leq c\left\|  b\right\|
_{BMO}\left\|  f\right\|  _{p},1<p<\infty, \label{uno}%
\end{equation}
where $[K,M_{b}]f=K(bf)-bK(f).$ Since each of the operators $f\rightarrow
K(bf)$ and $f\rightarrow bK(f)$ is unbounded on $L^{p},$ the cancellation that
results of taking their difference is essential for the validity of (\ref{uno}).

The Coifman-Rochberg-Weiss commutator theorem has found many applications in
the study of PDEs, Jacobians, Harmonic Analysis, and was also the starting
point of the Rochberg-Weiss \cite{rw} abstract theory of commutators in the
setting of scales of interpolation spaces, which itself has had many
applications (cf. \cite{ka}, \cite{ka1}, \cite{rr}, and the references therein).

It is instructive to review informally one of the proofs of (\ref{uno})
provided in \cite{crw}. Suppose that $b\in BMO,$ and fix $p>1.$ Then it is
well known that we can find $\varepsilon>0$ small enough such that, for all
$0<\alpha<\varepsilon,$ $e^{\alpha b}$ and $e^{-\alpha b}\in A_{p}$ (here
$A_{p}$ is the class of Muckenhoupt weights). Let $K$ be a CZ operator, then
$K$ is bounded on the weighted spaces $L^{p}(e^{\alpha b}),\left|
\alpha\right|  <\varepsilon.$ In other words, the family of operators
$f\rightarrow e^{\alpha b}K(e^{-\alpha b}f)$ is uniformly bounded on $L^{p}$
for $\left|  \alpha\right|  <\varepsilon.$ It follows readily that one can
extended these operators to an analytic family of operators $T(z)f=e^{zb}%
K(e^{-zb}f),$ for $\left|  z\right|  <\varepsilon,$ and then show that,
$\left.  \frac{d}{dz}T(z)f\right|  _{z=0}=\frac{1}{2}[K,M_{b}]f$ is also a
bounded operator on $L^{p}.$ In particular, it follows that, in the statement
of the theorem, we can replace CZ operators by operators $T$ with the same
weighted norm inequalities, i.e. the result holds for any operator $T,$ such
that for all weights in the $A_{p}$ class of Muckenhoupt, $T:L^{p}%
(w)\rightarrow L^{p}(w),1<p<\infty,$ boundedly.

The previous argument was the starting point of the Rochberg-Weiss \cite{rw}
theory of abstract commutator estimates for the complex method of
interpolation, later extended to the real method by these authors jointly with
Jawerth (cf. \cite{jrw}). The subject has been intensively developed in the
last 30 years (cf. the recent survey by Rochberg \cite{rr} and the references therein).

While the Rochberg-Weiss theory, when suitably specialized to weighted $L^{p}$
spaces, can be used to re-prove the Coifman-Rochberg-Weiss commutator theorem,
in this paper we consider a different problem: we give an abstract formulation
of the Coifman-Rochberg-Weiss commutator theorem which is valid for
interpolation scales themselves. Since we work with the real method, the
cancellations will be exploited via integration by parts and a suitable
re-interpretation of the relevant $BMO$ condition$.$

Before we formulate our main result let us recall some basic definitions
associated with the real method of interpolation (cf. \cite{bl} for more
details). Let $\bar{X}=(X_{0},X_{1})$ be a compatible pair of Banach spaces.
To define the real interpolation spaces\footnote{we shall only consider the
$J-$method in this note.} $(X_{0},X_{1})_{\theta,q}$ we start by considering
on $X_{0}\cap X_{1}$ the family of norms%
\[
J(t,x;\bar{X})=\max\{\left\|  x\right\|  _{X_{0}},t\left\|  x\right\|
_{X_{1}}\},t>0.
\]
Let $\theta\in(0,1),1\leq q\leq\infty.$ We consider the elements $f\in
X_{0}+X_{1},$ that can be represented by%
\[
f=\int_{0}^{\infty}u(s)\frac{ds}{s}\text{ (crucially here the convergence of
the integral is in the }X_{0}+X_{1}\text{ sense),}%
\]
where $u:(0,\infty)\rightarrow X_{0}\cap X_{1}.$ We let%
\[
\Phi_{\theta,q}(g)=\left\{  \int_{0}^{\infty}\left(  s^{-\theta}\left|
g(s)\right|  \right)  ^{q}\frac{ds}{s}\right\}  ^{1/q},
\]%
\[
\bar{X}_{\theta,q}=\{f=\int_{0}^{\infty}u(s)\frac{ds}{s}\text{ in }X_{0}%
+X_{1}:\Phi_{\theta,q}(J(s,u(s);\bar{X}))<\infty\},
\]%
\[
\left\|  f\right\|  _{\bar{X}_{\theta,q}}=\inf\{\Phi_{\theta,q}(J(s,u(s);\bar
{X})):f=\int_{0}^{\infty}u(s)\frac{ds}{s}\text{ in }X_{0}+X_{1}\}.
\]
Likewise, if $w$ is a positive function on $(0,\infty),$ we define the
corresponding spaces $\bar{X}_{\theta,q,w}$ by means of the use of the
function norm
\[
\Phi_{\theta,q,w}(g)=\Phi_{\theta,q}(wg).
\]

In this setting we consider the nonlinear operator%
\[
f\rightarrow u_{f}:(0,\infty)\rightarrow X_{0}\cap X_{1},
\]
where $u_{f}$ has been selected so that%
\begin{equation}
f=\int_{0}^{\infty}u_{f}(s)\frac{ds}{s}\text{ in }X_{0}+X_{1}, \label{intro}%
\end{equation}
and\footnote{we use $2$ for definitiness, obviously can replace $2$ by
$1+\varepsilon.$}%
\[
\Phi_{\theta,q}(J(s,u_{f}(s);\bar{X}))\leq2\left\|  f\right\|  _{\bar
{X}_{\theta,q}}.
\]
We then define
\begin{equation}
\Omega f=\Omega_{\bar{X}}f=\int_{0}^{\infty}u_{f}(s)\log s\frac{ds}{s}.
\label{dos}%
\end{equation}
The commutator theorem in this context (cf. \cite{jrw}) states that if
$T:\bar{X}\rightarrow\bar{Y}$ is a bounded linear operator, then the nonlinear
operator
\begin{align}
\left[  T,\Omega\right]  f  &  =T(\Omega_{\bar{X}}f)-\Omega_{\bar{Y}%
}(Tf)\nonumber\\
&  =\int_{0}^{\infty}(T(u_{f}(s))-u_{Tf}(s))\log s\frac{ds}{s} \label{tres}%
\end{align}
is bounded,%
\[
\left\|  \left[  T,\Omega\right]  f\right\|  _{\bar{Y}_{\theta,q}}\leq
c\left\|  T\right\|  _{\bar{X}\rightarrow\bar{Y}}\left\|  f\right\|  _{\bar
{X}_{\theta,q}}.
\]

One possible interpretation of the appearance of the logarithm in formula
(\ref{dos}) (and hence (\ref{tres})) can be given if we try to imitate the
arguments of Coifman-Rochberg-Weiss and bring into the argument analytic
functions with suitable cancellations. Indeed, if we represent the elements of
$\bar{X}_{\theta_{0},q}$ using the normalization $u_{\theta_{0}f}%
(s)=s^{\theta_{0}}u_{f}(s),$ then the elements in $\bar{X}_{\theta_{0},q}$ can
be represented by analytic functions (with appropriate control),%
\[
F(z)=\int_{0}^{\infty}s^{(z-\theta_{0})}(u_{\theta_{0}f}(s))\frac{ds}%
{s},\;F(\theta_{0})=f.
\]
In this setting we have%
\[
F^{\prime}(\theta_{0})=\Omega f.
\]

The crucial point of the cancellation argument is that, while operators
represented by derivatives of analytic functions can be unbounded (since we
may lose control of the norm estimates)$,$ the canonical representation of
$\left[  T,\Omega\right]  $%
\[
G^{\prime}(\theta_{0})=\left[  T,\Omega\right]  f,
\]
with%
\[
G(z)=\int_{0}^{\infty}s^{(z-\theta_{0})}(Tu_{\theta_{0}f}(s)-u_{\theta_{0}%
Tf}(s))\frac{ds}{s},
\]
exhibits the crucial cancellation%
\begin{align}
G(\theta_{0})  &  =\int_{0}^{\infty}(Tu_{\theta_{0}f}(s)-u_{\theta_{0}%
Tf}(s))\frac{ds}{s}\nonumber\\
&  =Tf-Tf\nonumber\\
&  =0, \label{amiga}%
\end{align}
which allows us to control the norm of $G^{\prime}(\theta_{0})$.

It is, of course, possible to eliminate all references to analytic functions,
and formulate the results in terms of representations that exhibit
cancellations. From this point of view the ``badness'' of the commutators is
expressed by the fact that their canonical representations have an extra
unbounded log factor (cf. (\ref{tres})) which would lead to the weaker
estimate%
\[
\left[  T,\Omega\right]  :\bar{X}_{\theta,q}\rightarrow\bar{Y}_{\theta
,q,\frac{1}{(1+\left|  \log s\right|  )}},\text{ (note that }\bar{Y}%
_{\theta,q}\subsetneqq\bar{Y}_{\theta,q,\frac{1}{(1+\left|  \log s\right|  )}%
}).
\]
Here is where the cancellation (\ref{amiga}), now expressed without reference
to analytic functions, simply as an integral equal to zero, comes to our
rescue and allows us to integrate by parts to find the ``better''
representation,
\begin{equation}
\left[  T,\Omega\right]  f=\int_{0}^{\infty}(\int_{0}^{t}(Tu_{\theta_{0}%
f}(s)-u_{\theta_{0}Tf}(s))\frac{ds}{s})\frac{ds}{s}, \label{buena}%
\end{equation}
which leads to the correct estimate%
\[
\left[  T,\Omega\right]  :\bar{X}_{\theta,q}\rightarrow\bar{Y}_{\theta,q}.
\]
This point of view was developed in \cite{mr}.

To formulate the Coifman-Rochberg-Weiss theorem in our setting we give a
different interpretation to the logarithm that appears in the formulae. First,
for a given weight $w$ we introduce the (possibly non linear) operators
$\Omega_{w},$ defined by%
\[
\Omega_{w}(f)=\int_{0}^{\infty}u_{f}(s)w(s)\frac{ds}{s}.
\]
It follows that for $w\in L^{\infty}(0,\infty),$ the corresponding $\Omega
_{w}$ is (trivially) a bounded operator,%
\[
\left\|  \Omega_{w}(f)\right\|  _{\bar{X}_{\theta,q}}\leq c\left\|  w\right\|
_{L^{\infty}}\left\|  f\right\|  _{\bar{X}_{\theta,q}},
\]
and therefore the corresponding commutators $\left[  T,\Omega_{w}\right]  $
are also bounded. On the other hand, for the mildly unbounded function
$w(s)=\log(s),$ we have $\Omega_{w}=\Omega,$ which is not bounded on $\bar
{X}_{\theta,q},$ but for which cancellations imply the boundedness of
commutators of the form $\left[  T,\Omega\right]  .$ Now, as is well known,
the logarithm is a typical example of a function with $BMO$ behavior.
Therefore we now ask more generally: for which weights $w$ can we assert that
for all bounded linear operators $T:\bar{X}\rightarrow\bar{Y},$ we have that
$\left[  T,\Omega_{w}\right]  $ is a bounded operator as well? The answer to
this question is what we shall call ``the abstract Coifman-Rochberg-Weiss theorem.''

Not surprisingly the answer is given in terms of a suitable $BMO$ type space
which allows us to control the oscillations of $w.$ Let $Pw(t)=\frac{1}{t}%
\int_{0}^{t}w(s)ds$ and define%
\[
w^{\#}(t)=Pw(t)-w(t)=\frac{1}{t}\int_{0}^{t}w(s)ds-w(t)=\frac{1}{t}\int
_{0}^{t}\left(  w(s)-w(t)\right)  ds.
\]
Then we consider the following analog\footnote{For martingales it can be
explicitly shown, by means of selecting appropriate sigma fields (cf.
\cite{dmy}), that $W$ is a $BMO$ martingale space. $W$ has also appeared
before in several papers on interpolation theory (cf. \cite{cws},
\cite{bmr1}).} of $BMO(R_{+})$ introduced in \cite{ms}:%
\[
W=\{w:w^{\#}(t)\in L^{\infty}(0,\infty)\},\text{ \ with \ }\left\|  w\right\|
_{W}=\left\|  Pw-w\right\|  _{L^{\infty}}.
\]
There is a direct connection between $W$ and the space $L(\infty,\infty)$ of
Bennett-DeVore-Sharpley \cite{bds}:%
\[
w\in L(\infty,\infty)\Leftrightarrow w^{\ast}\in W,
\]
where $w^{\ast}$ denotes the non-increasing rearrangement of $w.$ In
particular, we note that, as expected, the $\log$ has bounded oscillation
since%
\[
(\log t)^{\#}=\frac{1}{t}\int_{0}^{t}\log sds-\log t=-1.
\]

It will turn out that $W$ is the correct way to measure oscillation in our
context. In particular, we will show below that, when dealing with the
commutators $\left[  T,\Omega_{w}\right]  ,$\ the corresponding ``good
representation'' (cf. (\ref{buena})) is given by%
\[
\left[  T,\Omega_{w}\right]  f=\int_{0}^{\infty}(\int_{0}^{t}(Tu_{f}%
(s)-u_{Tf}(s))\frac{ds}{s})w^{\#}(s)\frac{ds}{s}.
\]

The purpose of this note is to prove the following abstract analog of the
Coifman-Rochberg-Weiss commutator theorem

\begin{theorem}
\label{teoA}Suppose that $w\in W,$ and let $\bar{X},\bar{Y},$ be Banach pairs.
Then, for any bounded linear operator $T:\bar{X}\rightarrow\bar{Y},$ the
commutator $\left[  T,\Omega_{w}\right]  $ is bounded, $\left[  T,\Omega
_{w}\right]  :\bar{X}_{\theta,q}\rightarrow\bar{Y}_{\theta,q},$ $0<\theta
<1,1\leq q\leq\infty,$ and, moreover,
\[
\left\|  \left[  T,\Omega_{w}\right]  f\right\|  _{\bar{Y}_{\theta,q}}\leq
c\left\|  T\right\|  _{\bar{X}\rightarrow\bar{Y}}\left\|  w\right\|
_{W}\left\|  f\right\|  _{\bar{X}_{\theta,q}}.
\]
\end{theorem}

We will also prove higher order versions of this result (cf. \cite{mr} and the
references therein). Using the strong form of the fundamental lemma (cf.
\cite{cjm}) one can connect the results above with those obtained in
\cite{bmr} for the $K-$method, and, moreover, give explicit instances of these operators.

\section{Representation Theorems}

As we have indicated in the Introduction, commutator theorems can be
formulated as results about special representations of certain elements in
interpolation scales. To develop our program explicitly it will be necessary
to integrate by parts often, so we start by collecting some elementary
calculations that will be useful for that purpose.

\begin{lemma}
\label{l2}The operator $P$ is bounded on $W.$
\end{lemma}

\begin{proof}%
\begin{align*}
(Pw)^{\#}(t)  &  =\frac{1}{t}\int_{0}^{t}Pw(s)ds-Pw(t)\\
&  =\frac{1}{t}\int_{0}^{t}\left(  Pw(s)-w(s)\right)  ds+Pw(t)-Pw(t).
\end{align*}
Therefore,%
\[
\left|  (Pw)^{\#}(t)\right|  \leq\left\|  w\right\|  _{W}.
\]
\end{proof}

\begin{lemma}
\label{l1} Let $w\in W$, and let $0<\theta<1.$ Then%
\[
\lim\limits_{t\rightarrow0}t^{\theta}w(t)=\lim\limits_{t\rightarrow\infty
}t^{-\theta}w(t)=0.
\]

\begin{proof}
Write $Pw=w^{\#}+w,$ then, since $w^{\#}$ is bounded, $\lim
\limits_{t\rightarrow0}t^{\theta}w^{\#}(t)=\lim\limits_{t\rightarrow\infty
}t^{-\theta}w^{\#}(t)=0,$ and we see that it is enough to show that
$\lim\limits_{t\rightarrow0}t^{\theta}Pw(t)=\lim\limits_{t\rightarrow\infty
}t^{-\theta}Pw(t)=0.$ Now, from $tPw(t)=\int_{0}^{t}w(s)ds,$ we get
$(Pw)^{\prime}(t)=-\frac{Pw(t)-w(t)}{t}.$ Therefore,%
\begin{align*}
\left|  Pw(t)\right|   &  \leq\left|  Pw(1)\right|  +\left|  \int_{t}%
^{1}w^{\#}(s)\frac{ds}{s}\right| \\
&  \leq\left\|  w\right\|  _{W}(1+\left|  \log t\right|  ).
\end{align*}
and the result follows.
\end{proof}
\end{lemma}

Although we shall not make use of the next result in this section it is
convenient to state it here to stress the $BMO$ characteristics of the space $W.$

\begin{lemma}
\label{l3} (i) (cf. \cite{bmr}) Let $\overline{Q}f(t)=\int_{t}^{1}%
f(s)\frac{ds}{s}$ then
\[
W=L_{\infty}+\overline{Q}(L_{\infty}).
\]
(ii) Let $W_{1}=\left\{  w:\sup\limits_{s}\left|  sw^{\prime}(s)\right|
<\infty\right\}  .$ Then, $W_{1}\subset W.$

(iii)%
\[
W=L_{\infty}+W_{1}.
\]
\end{lemma}

\begin{proof}
(i) see \cite{bmr}.

(ii) Suppose that $w\in W_{1}$. Integrating by parts
\[
\frac{1}{x}\int_{0}^{x}sw^{\prime}(s)ds=\frac{1}{x}\left.  sw(s)\right|
_{s=0}^{s=x}-\frac{1}{x}\int_{0}^{x}w(s)ds.
\]
It is easy to see (cf. Lemma \ref{l1}) that $\lim\limits_{x\rightarrow
0}sw(s)=0,$ hence%
\[
\left|  w^{\#}(x)\right|  =\left|  P(sw^{\prime}(s))(x)\right|  .
\]
Consequently, since $P$ is bounded on $L^{\infty},$ it follows that $w^{\#}\in
L^{\infty}$ and therefore $w\in W.$

(iii) Suppose that $w\in W.$ Since $\left|  t\left(  Pw\right)  ^{\prime
}\right|  =\left|  w^{\#}(t)\right|  ,$ it follows that $Pw\in W_{1}.$ The
desired decomposition is therefore given by%
\[
w=\underset{L^{\infty}}{\underbrace{\left(  w-Pw\right)  }}+\underset{W_{1}%
}{\underbrace{Pw}}.
\]
\end{proof}

The next result gives the representation theorem that we need to prove Theorem
\ref{teoA}.

\begin{theorem}
\label{t1}Let $\overline{H}=(H_{0},H_{1})$ be a Banach pair, and suppose that
$w\in W.$ Suppose that an element $f\in H_{0}+H_{1}$ can be represented as
\[
f=\int_{0}^{\infty}u(s)w(s)\frac{ds}{s},
\]
with
\[
\int_{0}^{\infty}u(s)\frac{ds}{s}=0,\text{ }\Phi_{\theta,q}(J(t,u(t);\overline
{H}))<\infty.
\]
Then,
\[
f\in\overline{H}_{\theta,q},
\]
and, moreover,%
\[
\left\|  f\right\|  _{\overline{H}_{\theta,q}}\leq c_{\theta,q}\left\|
w\right\|  _{W}\Phi_{\theta,q}(J(t,u(t);\overline{H})).
\]
\end{theorem}

\begin{proof}
Write%
\begin{align*}
f  &  =\int_{0}^{\infty}u(s)w(s)\frac{ds}{s}\\
&  =\int_{0}^{\infty}u(s)(w(s)-Pw(s))\frac{ds}{s}+\int_{0}^{\infty
}u(s)Pw(s)\frac{ds}{s}\\
&  =I_{1}+I_{2}.
\end{align*}
It is plain that%
\[
\left\|  I_{1}\right\|  _{\overline{H}_{\theta,q}}\leq\left\|  w\right\|
_{W}\Phi_{\theta,q}(J(t,u(t);\overline{H})).
\]
It remains to estimate $I_{2}.$ We integrate by parts:%
\[
I_{2}=\left.  Pw(t)\int_{0}^{t}u(s)\frac{ds}{s}\right]  _{0}^{\infty}-\int
_{0}^{\infty}\left(  \int_{0}^{t}u(s)\frac{ds}{s}\right)  (w(t)-Pw(t))\frac
{dt}{t}.
\]
The integrated term vanishes. Suppose first that $q>1.$ We can write%

\begin{align}
\left\|  \int_{0}^{t}u(s)\frac{ds}{s}\right\|  _{H_{0}}  &  \leq\int_{0}%
^{t}J(s,u(s))\frac{ds}{s}\leq\left(  \int_{0}^{t}\left(  \frac{J(s,u(s))}%
{s^{\theta}}\right)  ^{q}\frac{ds}{s}\right)  ^{1/q}\left(  \int_{0}%
^{t}s^{\theta q\prime}\frac{ds}{s}\right)  ^{1/q^{\prime}}\label{referato}\\
&  \leq c_{\theta,q}\Phi_{\theta,q}(J(t,u(t);\overline{H}))t^{\theta
}.\nonumber
\end{align}
By Lemma \ref{l2} $Pw\in W$ and therefore we may apply Lemma \ref{l1} to
conclude that%
\begin{align*}
\lim_{t\rightarrow0}\left|  Pw(t)\right|  \left\|  \int_{0}^{t}u(s)\frac
{ds}{s}\right\|  _{H_{0}}  &  \leq c_{\theta,q}\lim_{t\rightarrow0}%
\Phi_{\theta,q}(J(t,u(t);\overline{H}))t^{\theta}\left|  Pw(t)\right| \\
&  =0.
\end{align*}
Likewise, using the cancelation condition
\begin{equation}
\int_{0}^{t}u(s)\frac{ds}{s}=-\int_{t}^{\infty}u(s)\frac{ds}{s}, \label{can}%
\end{equation}
we have that
\[
\left\|  \int_{t}^{\infty}u(s)\frac{ds}{s}\right\|  _{H_{1}}\leq c\Phi
_{\theta,q}(J(t,u(t);\overline{H}))t^{-\theta},
\]
and once again we can apply Lemma \ref{l1} and find that%
\[
\lim_{t\rightarrow\infty}\left|  Pw(t)\right|  \left\|  \int_{t}^{\infty
}u(s)\frac{ds}{s}\right\|  _{H_{1}}=0.
\]

The case $q=1$ is simpler. For example, instead of using Holder's inequality
in (\ref{referato}) we write%
\[
\int_{0}^{t}J(s,u(s))\frac{ds}{s}=\frac{t^{\theta}}{t^{\theta}}\int_{0}%
^{t}J(s,u(s))\frac{ds}{s}\leq t^{\theta}\int_{0}^{t}\frac{J(s,u(s))}%
{s^{\theta}}\frac{ds}{s}.
\]

It remains to estimate the $\overline{H}_{\theta,q}$ norm of $I_{2}=\int
_{0}^{\infty}\left(  \int_{0}^{t}u(s)\frac{ds}{s}\right)  (w(t)-Pw(t))\frac
{dt}{t}.$ By definition,%
\begin{equation}
\left\|  I_{2}\right\|  _{\overline{H}_{\theta,q}}\leq\Phi_{\theta,q}(J(t)),
\label{laotra}%
\end{equation}
where
\begin{align*}
J(t)  &  =J(t,\left(  \int_{0}^{t}u(s)\frac{ds}{s}\right)
(w(t)-Pw(t));\overline{H})\\
&  \leq\left\|  w\right\|  _{W}\left(  \left\|  \int_{0}^{t}u(s)\frac{ds}%
{s}\right\|  _{H_{0}}+t\left\|  \int_{0}^{t}u(s)\frac{ds}{s}\right\|  _{H_{1}%
}\right)  .
\end{align*}
The first term on the right hand side can be estimated directly by Minkowski's
inequality%
\[
\left\|  \int_{0}^{t}u(s)\frac{ds}{s}\right\|  _{H_{0}}\leq\int_{0}%
^{t}J(s,u(s);\overline{H})\frac{ds}{s},
\]

while for the second we argue that, by (\ref{can}),%
\begin{align*}
t\left\|  \int_{0}^{t}u(s)\frac{ds}{s}\right\|  _{H_{1}}  &  =t\left\|
\int_{t}^{\infty}u(s)\frac{ds}{s}\right\|  _{H_{1}}\\
&  \leq t\int_{t}^{\infty}J(s,u(s);\overline{H})\frac{ds}{s^{2}}.
\end{align*}
Altogether, we arrive at
\[
J(t)\leq\left\|  w\right\|  _{W}\left(  \int_{0}^{t}J(s,u(s);\overline
{H})\frac{ds}{s}+t\int_{t}^{\infty}J(s,u(s);\overline{H})\frac{ds}{s^{2}%
}\right)  .
\]
Therefore, applying the $\Phi_{\theta,q}$ norm on both sides of the previous
inequality and then using Hardy's inequalities to estimate the right hand
side, we get%
\begin{equation}
\Phi_{\theta,q}(J(t))\leq c_{\theta,q}\left\|  w\right\|  _{W}\Phi_{\theta
,q}\left(  J(t,u(t);\overline{H})\right)  . \label{launa}%
\end{equation}
Combining (\ref{launa}) and (\ref{laotra})%
\[
\left\|  I_{2}\right\|  _{\overline{H}_{\theta,q}}\leq c_{\theta,q}\left\|
w\right\|  _{W}\Phi_{\theta,q}(J(t)),
\]
and collecting the estimates for $I_{1}$ and $I_{2}$ we finally obtain
\[
\left\|  f\right\|  _{\overline{H}_{\theta,q}}\leq c_{\theta,q}\left\|
w\right\|  _{W}\Phi_{\theta,q}(J(t,u(t);\overline{H}))
\]
as we wished to show.
\end{proof}

We are now ready for the proof of Theorem \ref{teoA}.

\begin{proof}
Suppose that $T$ is a given bounded linear operator $T:\bar{X}\rightarrow
\bar{Y},$ and let $w\in W.$ Let $\tilde{u}(t)=((u_{Tf}(t)-T(u_{f}(t)).$ Then
\[
\left[  T,\Omega_{w}\right]  f=\int_{0}^{\infty}\tilde{u}(t)w(t)\frac{dt}{t}%
\]
with%
\[
\Phi_{\theta,q}(J(t,\tilde{u}(t);\bar{Y})\leq c\left\|  T\right\|  _{\bar
{X}\rightarrow\bar{Y}}\left\|  f\right\|  _{\overline{X}_{\theta,q}}.
\]
Since, moreover,
\[
\int_{0}^{\infty}\tilde{u}(t)\frac{dt}{t}=0,
\]
we can apply theorem \ref{t1} to conclude that%
\[
\left\|  \left[  T,\Omega_{w}\right]  f\right\|  _{\overline{Y}_{\theta,q}%
}\leq c\left\|  w\right\|  _{W}\left\|  T\right\|  _{\bar{X}\rightarrow\bar
{Y}}\left\|  f\right\|  _{\overline{X}_{\theta,q}},
\]
as we wished to show.
\end{proof}

\section{Higher order cancellations}

We adapt the analysis of \cite{mr} to handle higher order cancellations. The
corresponding higher order commutator theorems that follow will be stated and
proved in the next section.

\begin{theorem}
\label{t2} Let $\overline{H}$ be a Banach pair, and let $w\in W.$ Suppose that
$f$ admits a representation
\[
f=\int_{0}^{\infty}u(s)\left(  Pw(s)\right)  ^{2}\frac{ds}{s},
\]
with
\[
\int_{0}^{\infty}u(s)\frac{ds}{s}=0,\text{ }\int_{0}^{\infty}u(s)Pw(s)\frac
{ds}{s}=0;\text{ \ }\Phi_{\theta,q}(J(t,u(t);\overline{H}))<\infty
\]
then,
\[
f\in\overline{H}_{\theta,q},
\]
and, moreover,%
\[
\left\|  f\right\|  _{\overline{H}_{\theta,q}}\leq c\left\|  w\right\|
_{W}^{2}\Phi_{\theta,q}(J(t,u(t);\overline{H})).
\]

\begin{proof}
We will integrate by parts repeatedly. We start writing%
\[
f=\int_{0}^{\infty}u(t)\left(  Pw(t)\right)  ^{2}\frac{dt}{t}=\int_{0}%
^{\infty}Pw(t)d\left(  \int_{0}^{t}u(s)Pw(s)\frac{ds}{s}\right)  .
\]
Then,
\[
f=\left.  Pw(t)\int_{0}^{t}u(s)Pw(s)\frac{ds}{s}\right]  _{0}^{\infty}%
-\int_{0}^{\infty}\left(  \int_{0}^{t}u(s)Pw(s)\frac{ds}{s}\right)
(w(t)-Pw(t))\frac{dt}{t},
\]
we will show below that the integrated term vanishes, then%
\begin{equation}
f=-\int_{0}^{\infty}\left(  \int_{0}^{t}u(s)Pw(s)\frac{ds}{s}\right)
(w(t)-Pw(t))\frac{dt}{t}.\label{insertada}%
\end{equation}
Now we consider the inner integral\ and integrate by parts%
\[
\int_{0}^{t}u(s)Pw(s)\frac{ds}{s}=\int_{0}^{t}Pw(s)d\left(  \int_{0}%
^{s}u(r)\frac{dr}{r}\right)  ,
\]
using the fact that (cf. the proof of Theorem \ref{t1}) $\lim
\limits_{t\rightarrow0}Pw(t)\int_{0}^{t}u(s)\frac{ds}{s}=0,$ we get%
\[
\int_{0}^{t}u(s)Pw(s)\frac{ds}{s}=Pw(t)\int_{0}^{t}u(s)\frac{ds}{s}-\int
_{0}^{t}\left(  \int_{0}^{r}u(s)\frac{ds}{s}\right)  \left(
w(r)-Pw(r)\right)  \frac{dr}{r}.
\]
Inserting this result back in (\ref{insertada}) we find that%
\begin{align*}
f &  =-\int_{0}^{\infty}\left(  \int_{0}^{t}u(s)Pw(s)\frac{ds}{s}\right)
(w(t)-Pw(t))\frac{dt}{t}\\
&  =\int_{0}^{\infty}\left(  Pw(t)\int_{0}^{t}u(s)\frac{ds}{s}\right)
w^{\#}(t)\frac{dt}{t}+\int_{0}^{\infty}\left(  \int_{0}^{t}\left(  \int
_{0}^{r}u(s)\frac{ds}{s}\right)  w^{\#}(r)\frac{dr}{r}\right)  w^{\#}%
(t)\frac{dt}{t}\\
&  =I_{0}+I_{1}.
\end{align*}
Integrating by parts $I_{0}$ we get%
\begin{align*}
I_{0} &  =\left.  Pw(t)\int_{0}^{t}\left(  w^{\#}(r)\int_{0}^{r}u(s)\frac
{ds}{s}\right)  \frac{dr}{r}\right|  _{0}^{\infty}\\
&  +\int_{0}^{\infty}\left(  \int_{0}^{t}\left(  \int_{0}^{r}u(s)\frac{ds}%
{s}\right)  w^{\#}(r)\frac{dr}{r}\right)  w^{\#}(t)\frac{dt}{t},
\end{align*}
where once again the integrated term vanishes. Hence,%
\[
I_{0}=I_{1}.
\]
Therefore, if we let $U(t)=2\left(  \int_{0}^{t}\left(  \int_{0}^{r}%
u(s)\frac{ds}{s}\right)  w^{\#}(r)\frac{dr}{r}\right)  w^{\#}(t),$ $f$ can be
represented by%
\[
f=\int_{0}^{\infty}U(t)\frac{dt}{t}.
\]
Now we estimate the corresponding $J-$functional, $J(t)=J(t,U(t);\bar{H}),$
by
\begin{align*}
&  2\left\|  w\right\|  _{W}\left(  \left\|  \int_{0}^{t}\left(  \int_{0}%
^{r}u(s)\frac{ds}{s}\right)  w^{\#}(r)\frac{dr}{r}\right\|  _{H_{0}}+t\left\|
\int_{0}^{t}\left(  \int_{0}^{r}u(s)\frac{ds}{s}\right)  w^{\#}(r)\frac{dr}%
{r}\right\|  _{H_{1}}\right)  \\
&  =2\left\|  w\right\|  _{W}\left(  C_{0}+tC_{1}\right)  .
\end{align*}
We readily see that $C_{0}$ is majorized by%
\[
C_{0}\leq\left\|  w\right\|  _{W}\int_{0}^{t}\left(  \int_{0}^{r}%
J(s,u(s);\overline{H})\frac{ds}{s}\right)  \frac{dr}{r}=\left\|  w\right\|
_{W}\int_{0}^{t}J(r,u(r);\overline{H})\ln\frac{t}{r}\frac{dr}{r}.
\]
To handle $C_{1}$ we work with the integral inside the norm $H_{1}$ by first
using $\int_{0}^{r}u(s)\frac{ds}{s}=-\int_{r}^{\infty}u(s)\frac{ds}{s}$ and
then changing the order of integration. We find that%
\[
C_{1}=\left\|  \lim_{\alpha\rightarrow0}C(\alpha)\right\|  _{H_{1}},
\]
where $C(\alpha)=\int_{0}^{t}\int_{\alpha}^{s}w^{\#}(r)\frac{dr}{r}%
u(s)\frac{ds}{s}+\int_{t}^{\infty}\int_{\alpha}^{t}w^{\#}(r)\frac{dr}%
{r}u(s)\frac{ds}{s}.$ We compute $C(\alpha)$ using the formula $(Pw)^{\prime
}(t)=-\frac{w^{\#}(t)}{t},$ and we get
\[
C(\alpha)=Pw(\alpha)\int_{0}^{t}u(s)\frac{ds}{s}-\int_{0}^{t}Pw(s)u(s)\frac
{ds}{s}+Pw(\alpha)\int_{t}^{\infty}u(s)\frac{ds}{s}-Pw(t)\int_{t}^{\infty
}u(s)\frac{ds}{s}.
\]
Now by the cancellation conditions:%
\[
\int_{0}^{\infty}u(s)\frac{ds}{s}=0\Longrightarrow Pw(\alpha)\int_{0}%
^{t}u(s)\frac{ds}{s}=-Pw(\alpha)\int_{t}^{\infty}u(s)\frac{ds}{s},
\]
and%
\[
\int_{0}^{\infty}u(s)Pw(s)\frac{ds}{s}=0\Longrightarrow\int_{0}^{t}%
u(s)Pw(s)\frac{ds}{s}=-\int_{t}^{\infty}u(s)Pw(s)\frac{ds}{s},
\]
we have%
\begin{align*}
C(\alpha) &  =\int_{t}^{\infty}u(s)[Pw(s)-Pw(t)]\frac{ds}{s}\\
&  =\int_{t}^{\infty}u(s)\int_{t}^{s}w^{\#}(r)\frac{dr}{r}\frac{ds}{s}.
\end{align*}
All in all it follows that,%
\begin{align*}
C_{1} &  \leq\left\|  w\right\|  _{W}\int_{t}^{\infty}\left\|  u(s)\right\|
_{H_{1}}\ln\frac{s}{t}\frac{ds}{s}\\
&  \leq\left\|  w\right\|  _{W}\int_{t}^{\infty}J(s,u(s);\overline{H})\ln
\frac{s}{t}\frac{ds}{s^{2}}.
\end{align*}
Summarizing,%
\[
J(t)\leq2\left\|  w\right\|  _{W}^{2}\left(  \int_{0}^{t}J(r,u(r);\overline
{H})\ln\frac{t}{r}\frac{dr}{r}+t\int_{t}^{\infty}J(r,u(r);\overline{H}%
)\ln\frac{r}{t}\frac{dr}{r^{2}}\right)  .
\]
Applying the $\Phi_{\theta,q}$ norm and Hardy's inequalities (twice) we
finally obtain%
\begin{align*}
\left\|  f\right\|  _{\overline{H}_{\theta,q}} &  \leq c\Phi_{\theta,q}\left(
J(t,u(t);\overline{H})\right)  \\
&  \leq c\left\|  w\right\|  _{W}^{2}\Phi_{\theta,q}\left(  J(t,u(t);\overline
{H})\right)  .
\end{align*}
To conclude the proof it remains to verify that the integrated terms we have
collected along the way effectively vanish. More precisely, it remains to
prove that
\begin{equation}
\lim_{t\rightarrow\xi}Pw(t)\int_{0}^{t}u(s)Pw(s)\frac{ds}{s}=0\text{, \ for
}\xi=0,\infty,\label{cer1}%
\end{equation}
and
\begin{equation}
\lim_{t\rightarrow\xi}Pw(t)\int_{0}^{t}\left(  \left(  w(r)-Pw(r)\right)
\int_{0}^{r}u(s)\frac{ds}{s}\right)  \frac{dr}{r}=0\text{, for }\xi
=0,\infty\text{.}\label{cer2}%
\end{equation}
To handle these limits we shall assume that $q>1$, the case $q=1$ is easier
(cf. the proof of Theorem \ref{t1} above). We start with (\ref{cer1}):%
\begin{align*}
\left\|  \int_{0}^{t}u(s)Pw(s)\frac{ds}{s}\right\|  _{H_{0}} &  \leq\int
_{0}^{t}J(s,u(s);\overline{H})\left|  Pw(s)\right|  \frac{ds}{s}\\
&  \leq\left(  \int_{0}^{t}\left(  \frac{J(s,u(s);\overline{H})}{s^{\theta}%
}\right)  ^{q}\frac{ds}{s}\right)  ^{1/q}\left(  \int_{0}^{t}\left(
s^{\theta}\left|  Pw(s)\right|  \right)  ^{q^{\prime}}\frac{ds}{s}\right)
^{1/q^{\prime}}\\
&  \leq\left(  \Phi_{\theta,q}(J(s,u(s);\overline{H}))\right)  c\left(
\int_{0}^{t}\left(  s^{\theta}\left|  w(s)\right|  \right)  ^{q^{\prime}}%
\frac{ds}{s}\right)  ^{1/q^{\prime}}\text{ (by Hardy's inequality)}%
\end{align*}
Let $\widetilde{\theta}$ $>0$ be such that $\theta-\widetilde{\theta}$ $>0.$
Since $w\in W$ $\Rightarrow Pw\in W$ (cf. Lemma \ref{l2})$,$ therefore, by
Lemma \ref{l1}, we have
\[
\left|  t^{\widetilde{\theta}}Pw(t)\right|  \leq1\text{ \ \ \ (if }t\text{
suff. close to }0\text{).}%
\]
Thus, for small t,
\[
\left(  \int_{0}^{t}\left(  s^{\theta}\left|  Pw(s)\right|  \right)
^{q^{\prime}}\frac{ds}{s}\right)  ^{1/q^{\prime}}\leq\left(  \int_{0}%
^{t}\left(  s^{\theta-\widetilde{\theta}}\right)  ^{q^{\prime}}\frac{ds}%
{s}\right)  ^{1/q^{\prime}}\leq ct^{\theta-\widetilde{\theta}},
\]
and
\[
\lim_{t\rightarrow0}\left\|  Pw(t)\int_{0}^{t}u(s)w(s)\frac{ds}{s}\right\|
_{H_{0}}\leq\lim_{t\rightarrow0}ct^{\theta-\widetilde{\theta}}\left|
Pw(t)\right|  =0.
\]
The corresponding limit when $t\rightarrow\infty$ can be handled by the same
argument if we first use the cancellation property $\int_{0}^{t}%
u(s)Pw(s)\frac{ds}{s}=-\int_{t}^{\infty}u(s)Pw(s)\frac{ds}{s}$ and then apply
the $H_{1}$ norm.

To see (\ref{cer2}) we note that%
\begin{align*}
&  \left|  Pw(t)\right|  \left\|  \int_{0}^{t}\left(  w(r)-Pw(r)\right)
\int_{0}^{r}u(s)\frac{ds}{s}\frac{dr}{r}\right\|  _{H_{0}}\\
&  \leq\left\|  w\right\|  _{W}\left|  Pw(t)\right|  \int_{0}^{t}%
J(s,u(s);\overline{H})\ln\frac{t}{s}\frac{ds}{s}\\
&  \leq\left\|  w\right\|  _{W}\left|  Pw(t)\right|  t^{\theta}\left(
\Phi_{\theta,q}(J(s,u(s);\overline{H}))\right)  t^{-\theta}\left(  \int
_{0}^{t}\left(  s^{\theta}\ln\frac{t}{s}\right)  ^{q^{\prime}}\frac{ds}%
{s}\right)  ^{1/q^{\prime}}.
\end{align*}
Now, the term on the right hand side converges to zero when $t\rightarrow0$ by
Lemma \ref{l1} and the fact that near zero, $t^{-\theta}\left(  \int_{0}%
^{t}\left(  s^{\theta}\ln\frac{t}{s}\right)  ^{q^{\prime}}\frac{ds}{s}\right)
^{1/q^{\prime}}\leq t^{-\theta}\left(  \int_{0}^{t}\left(  s^{\theta}\frac
{s}{t}\right)  ^{q^{\prime}}\frac{ds}{s}\right)  ^{1/q^{\prime}}\leq
ct^{-\theta}t^{-1}t^{1+\theta}.$ Again the case $t\rightarrow\infty$ is
reduced to the case $t\rightarrow0$ by a familiar argument using cancellations.
\end{proof}
\end{theorem}

\begin{corollary}
\label{c1} Let $\overline{H}$ be a Banach pair, and let $w\in W.$ Suppose
that
\[
f=\int_{0}^{\infty}u(s)\left(  w(s)\right)  ^{2}\frac{ds}{s},
\]
with
\begin{equation}
\int_{0}^{\infty}u(s)\frac{ds}{s}=0,\text{ }\int_{0}^{\infty}u(s)w(s)\frac
{ds}{s}=0,\text{ }\int_{0}^{\infty}u(s)Pw(s)\frac{ds}{s}=0; \label{cancela}%
\end{equation}
and
\[
\Phi_{\theta,q}(J(t,u(t);\overline{H}))<\infty.
\]
Then,
\[
f\in\overline{H}_{\theta,q},
\]
and, moreover,
\[
\left\|  f\right\|  _{\overline{H}_{\theta,q}}\leq c\left\|  w\right\|
_{W}^{2}\Phi_{\theta,q}(J(t,u(t);\overline{H})).
\]
\end{corollary}

\begin{proof}
Write
\[
\int_{0}^{\infty}u(s)\left(  w(s)\right)  ^{2}\frac{ds}{s}=\int_{0}^{\infty
}u(s)\left(  w(s)-Pw(s)\right)  w(s)\frac{ds}{s}+\int_{0}^{\infty
}u(s)w(s)Pw(s)\frac{ds}{s}.
\]
Since $w(t)Pw(t)=\frac{\left(  w(t)\right)  ^{2}+\left(  Pw(t)\right)
^{2}-(w(t)-Pw(t))^{2}}{2},$ we have
\begin{align*}
\int_{0}^{\infty}u(s)\left(  w(s)\right)  ^{2}\frac{ds}{s}  &  =2\int
_{0}^{\infty}u(s)\left(  w(s)-Pw(s)\right)  w(s)\frac{ds}{s}-\int_{0}^{\infty
}u(s)\left(  w(s)-Pw(s)\right)  ^{2}\frac{ds}{s}\\
&  +\int_{0}^{\infty}u(s)\left(  Pw(s)\right)  ^{2}\frac{ds}{s}%
\end{align*}

We now show how to control each of these terms. Let $\tilde{u}(s)=u(s)\left(
w(s)-Pw(s)\right)  ,$ by the cancellation conditions (\ref{cancela}) it
follows that $\int_{0}^{\infty}\tilde{u}(s)\frac{ds}{s}=0.$ Therefore we can
apply Theorem \ref{t1} to conclude that $\int_{0}^{\infty}\tilde
{u}(s)w(s)\frac{ds}{s}$ $\in\overline{H}_{\theta,q}.$ It follows that%
\[
\left\|  2\int_{0}^{\infty}u(s)\left(  w(s)-Pw(s)\right)  w(s)\frac{ds}%
{s}\right\|  _{\overline{H}_{\theta,q}}\leq c\left\|  w\right\|  _{W}^{2}%
\Phi_{\theta,q}(J(t,u(t);\overline{H})).
\]
The second term is also under control since $\left(  w(s)-Pw(s)\right)  ^{2}$
is bounded. Finally we may apply Theorem \ref{t2} to control the remaining term.
\end{proof}

\begin{theorem}
Let $\overline{H}$ be a Banach pair, and let $w_{0},w_{1}\in W.$ Suppose that
\[
f=\int_{0}^{\infty}u(s)w_{0}(s)w_{1}(s)\frac{ds}{s},
\]
with
\[
\int_{0}^{\infty}u(s)\frac{ds}{s}=0,\text{ }\int_{0}^{\infty}u(s)w_{j}%
(s)\frac{ds}{s}=0,\text{ }\int_{0}^{\infty}u(s)Pw_{j}(s)\frac{ds}%
{s}=0\text{\ \ }(j=0,1)
\]
and
\[
\Phi_{\theta,q}(J(t,u(t);\overline{H}))<\infty.
\]
Then,
\[
f\in\overline{H}_{\theta,q}%
\]
and, moreover,%
\[
\left\|  f\right\|  _{\overline{H}_{\theta,q}}\leq c\max\{\left\|
w_{0}\right\|  _{W},\left\|  w_{1}\right\|  _{W}\}^{2}\Phi_{\theta
,q}(J(t,u(t);\overline{H})).
\]
\end{theorem}

\begin{proof}
Write
\[
w_{0}(s)w_{1}(s)=\frac{(w_{0}(s)+w_{1}(s))^{2}-w_{0}(s)^{2}-w_{1}(s)^{2}}{2}%
\]
and apply Corollary \ref{c1}.
\end{proof}

For $n>2$ we proceed by induction and we obtain

\begin{theorem}
Let $\overline{H}$ be a Banach pair, and let $w\in W$.

(i) Suppose that
\[
f=\int_{0}^{\infty}u(s)\left(  Pw(s)\right)  ^{n}\frac{ds}{s},
\]
with
\[
\int_{0}^{\infty}u(s)w(s)^{k}\frac{ds}{s}=0,\text{ }\int_{0}^{\infty
}u(s)Pw(s)^{k}\frac{ds}{s}=0,\text{ \ \ }(k=0,\cdots,n-1);
\]
and
\[
\Phi_{\theta,q}(J(t,u(t);\overline{H}))<\infty.
\]
Then,
\[
f\in\overline{H}_{\theta,q},
\]
and, moreover,%
\[
\left\|  f\right\|  _{\overline{H}_{\theta,q}}\leq c\Phi_{\theta
,q}(J(t,u(t);\overline{H})).
\]
(ii) If
\[
f=\int_{0}^{\infty}u(s)\left(  w(s)\right)  ^{n}\frac{ds}{s},
\]
with
\[
\int_{0}^{\infty}u(s)w(s)^{k}\frac{ds}{s}=0,\text{ }\int_{0}^{\infty
}u(s)Pw(s)^{k}\frac{ds}{s}=0,\text{ }\int_{0}^{\infty}u(s)w(s)^{n-k}%
Pw(s)^{k}\frac{ds}{s}=0,\text{ \ ,}(k=0,\cdots,n-1);
\]
and
\[
\Phi_{\theta,q}(J(t,u(t);\overline{H}))<\infty,
\]
then,
\[
f\in\overline{H}_{\theta,q},
\]
and, moreover,
\[
\left\|  a\right\|  _{\overline{H}_{\theta,q}}\leq c\left\|  w\right\|
_{W}^{n}\Phi_{\theta,q}(J(t,u(t);\overline{H})).
\]
\end{theorem}

\begin{remark}
In the classical case (cf. \cite{mr}, theorem 3) $w(t)=\ln t,$ and therefore
$Pw(t)=\ln t-1.$ Consequently the conditions
\[
\int_{0}^{\infty}u(s)Pw(s)^{k}\frac{ds}{s}=0,\int_{0}^{\infty}u(s)w(s)^{n-k}%
Pw(s)^{k}\frac{ds}{s}=0,\,(k=0,\cdots,n-1);
\]
actually follow from
\[
\int_{0}^{\infty}u(s)\left(  w(s)\right)  ^{k}\frac{ds}{s}=0,\;(k=0,\cdots
,n-1).
\]
\end{remark}

\section{Higher order commutators}

We consider higher order commutators defined as follows (cf. \cite{mr},
\cite{bmr}, \cite{rr}). Let $\bar{X}$ and $\bar{Y}$ be Banach pairs, and let
$T:\bar{X}\rightarrow$ $\bar{Y}$ be a bounded linear operator. Given a nearly
optimal representation (cf. \ref{intro} above)%
\[
f=\int_{0}^{\infty}u_{f}(s)\frac{ds}{s}%
\]
we let
\[
\Omega_{n,w}f=\frac{1}{n!}\int_{0}^{\infty}u_{f}(s)(w(s))^{n}\frac{ds}%
{s},\,n=0,1,...
\]
and form the commutators%
\[
C_{n,w}f=\left\{
\begin{array}
[c]{cc}%
Tf & ,\text{ }n=0\\
\left[  T,\Omega_{1,w}\right]  f & ,\ n=1\\
\left[  T,\Omega_{2,w}\right]  f-\Omega_{1,w}(C_{1,w}f) & ,\text{ }n=2\\
............ & \\
\left[  T,\Omega_{n,w}\right]  f-\Omega_{1,w}(C_{n-1,w}f)-\cdots\Omega
_{n-1,w}(C_{1,w}f) &
\end{array}
\right.
\]
Observe that the commutators $\left[  T,\Omega_{n,w}\right]  $ alone are not
bounded and we need to form more complicated expressions like $C_{n,w}$ in
order to produce the necessary cancellations. Moreover, since the operations
$\Omega_{j,w}$ are not linear, simple minded iterations of the form
$\Omega_{1,w}\left[  T,\Omega_{1,w}\right]  -\left[  T,\Omega_{1,w}\right]
\Omega_{1,w},etc,$ cannot be treated directly using Theorem \ref{teoA}.

\begin{theorem}
Suppose that $w\in W.$ Then the commutators $C_{n,w}$ are bounded, $C_{n,w}$
$:\bar{X}_{\theta,q}\rightarrow\bar{Y}_{\theta,q},$ $0<\theta<1,1\leq
q\leq\infty,$ and, moreover, for each instance $g=w,$ or $g=Pw,$ we have
\[
\left\|  C_{n,w}f\right\|  _{\bar{Y}_{\theta,q}}\leq c\left\|  T\right\|
_{\bar{X}\rightarrow\bar{Y}}\left\|  w\right\|  _{W}^{n}\left\|  f\right\|
_{\bar{X}_{\theta,q}}.
\]

\begin{proof}
We only consider in detail the case $n=2.$ Writing $w=(w-Pw)+Pw,$ we see that
we only need to deal with the commutator $C_{2,Pw}$. Let
\[
u(s)=T(u_{f}(s))-u_{T(f)}(s)
\]
then
\[
C_{2,Pw}(Tf)=\frac{1}{2}\int_{0}^{\infty}u(t)(Pw(t))^{2}\frac{dt}{t}-\int
_{0}^{\infty}\widetilde{u}(t)Pw(t)\frac{dt}{t},
\]
with
\[
\int_{0}^{\infty}\widetilde{u}(t)\frac{dt}{t}=\int_{0}^{\infty}u(t)Pw(t)\frac
{dt}{t};\int_{0}^{\infty}u(t)\frac{dt}{t}=0,
\]
and
\[
\Phi_{\theta,q}(J(t,\widetilde{u}(t),\overline{X}))\leq c\left\|  w\right\|
_{W}\left\|  f\right\|  _{\bar{X}_{\theta,q}}%
\]%

\[
\Phi_{\theta,q}(J(t,u(t),\overline{X}))\leq c\left\|  w\right\|  _{W}\left\|
f\right\|  _{\bar{X}_{\theta,q}}%
\]

Since%
\begin{align*}
\frac{1}{2}\int_{0}^{\infty}u(t)(Pw(t))^{2}\frac{dt}{t}  &  =\frac{1}{2}%
\int_{0}^{\infty}(Pw(t))^{2}d\left(  \int_{0}^{t}u(s)\frac{ds}{s}\right) \\
&  =\int_{0}^{\infty}\left(  \int_{0}^{t}u(s)\frac{ds}{s}\right)
Pw(t)w^{\#}(t)\frac{dt}{t},
\end{align*}
it follows that if we let%
\[
v(t)=(\int_{0}^{t}u(s)\frac{ds}{s})w^{\#}(t)
\]
then
\[
C_{2,Pw}(Tf)=\int_{0}^{\infty}(v(t)-\widetilde{u}(t))Pw(t)\frac{dt}{t},
\]
and%
\[
\int_{0}^{\infty}(v(t)-\widetilde{u}(t))\frac{dt}{t}=0.
\]
then theorem \ref{t1} implies that
\begin{align*}
\left\|  C_{2,Pw}(Tf)\right\|  _{\bar{Y}_{\theta,q}}  &  \leq c\left\|
w\right\|  _{W}\Phi_{\theta,q}(J(t,u(t);\bar{X}))+c\Phi_{\theta,q}%
(J(t,\widetilde{u}(t);\bar{X}))\\
&  \leq c\left\|  w\right\|  _{W}^{2}\left\|  f\right\|  _{\bar{X}_{\theta,q}%
}.
\end{align*}
as we wished to show.
\end{proof}
\end{theorem}

\section{Comparison with earlier results and some questions}

This paper was originally conceived in 1999-2000, when the first named author
spent one year in the Tropics. So publication was delayed somewhat and in the
mean time several papers on the subject have appeared. In particular,
\cite{mf} has similar statements framed in terms of weights of the form%
\begin{equation}
w(t)=\phi(\log t),\text{ with }\phi\text{ Lipchitz.} \label{forma}%
\end{equation}
One recognizes that these weights are included in our theory since for $w$ of
the form (\ref{forma}) we have (cf. Lemma \ref{l3} above)%
\[
\left\|  w\right\|  _{W_{1}}=\sup\left|  tw^{\prime}(t)\right|  =\left\|
\phi^{\prime}\right\|  _{\infty}<\infty.
\]
There is also a connection with \cite{bmr} (a longer version of this paper was
originally circulated in 1996 (cf. \cite{bmr1})). These papers emphasize the
connection between weighted norm inequalities, commutators and BMO type
conditions using the $K-$method, and $BMO$ conditions are formulated in terms
of properties of weights. Recall that for the $K-$method of interpolation we
define the corresponding $\Omega$ operations by%
\[
\Omega^{K}f=\int_{0}^{1}x_{0}(t)\frac{dt}{t}-\int_{1}^{\infty}x_{1}%
(t)\frac{dt}{t},
\]
or, more generally, by%
\[
\Omega_{w}^{K}f=\int_{0}^{1}x_{0}(t)w(t)\frac{dt}{t}-\int_{1}^{\infty}%
x_{1}(t)w(t)\frac{dt}{t},
\]
where%
\[
f=x_{0}(t)+x_{1}(t),\text{ and }\left\|  x_{0}(t)\right\|  _{H_{0}}+t\left\|
x_{1}(t)\right\|  _{H_{0}}\leq cK(t,f;\bar{H}).
\]
Using the strong form of the fundamental lemma of interpolation theory (cf.
\cite{cjm}) we can arrange to have $f=\int_{0}^{\infty}u_{f}(s)\frac{ds}{s},$
and%
\[
x_{0}(t)=\int_{0}^{t}u_{f}(s)\frac{ds}{s},x_{1}(t)=\int_{t}^{\infty}%
u_{f}(s)\frac{ds}{s}.
\]
It formally follows that%
\[
\Omega_{w}^{K}f=-\Omega_{Gw}f,
\]
where%
\[
Gw(s)=\int_{1}^{s}w(r)\frac{dr}{r}.
\]
In particular, if $w=1,$ then $Gw(s)=\log s$. Also note that%
\[
\sup_{s}\left|  s(Gw)^{\prime}(s)\right|  =\left\|  w\right\|  _{\infty}.
\]

Now a brief attempt to informally connect our work with Dan Waterman's
classical Fourier analysis. One source of inspiration for the formulation of
some of the results in this paper comes from the Littlewood-Paley theory,
framed in terms of semigroups, e.g. as developed in Stein \cite{st}. In the
abstract theory of Stein \cite{st} (cf. \cite{st} pag 121) the relevant
semigroups are represented, using the spectral theorem, by%
\[
T^{t}=\int_{0}^{\infty}e^{-\lambda t}dE(\lambda),
\]
and one considers (multiplier) operators of the form%
\[
T_{w}f=\int_{0}^{\infty}e^{-\lambda t}w(t)dE(\lambda)f,
\]
with $w\in L^{\infty}.$ The conclusion is that the operator $T_{(Lw)^{\prime}%
}$ is bounded on $L^{p},1<p<\infty,$ where%
\[
Lw(\lambda)=\int_{0}^{\infty}e^{-\lambda t}w(t)dt.
\]
We hope to come back to explore this subject elsewhere.

We conclude with a few suggestions for future explorations on related topics.

T1. One can formulate iterations of the operation $\#$ (cf. \cite{cws}) and
ask for its relevance in the study of higher order commutators.

T2. Despite several results (cf. \cite{cjmr}, \cite{bmr1}, \cite{rr}) one
feels that the duality theory associated to the $\Omega$ operators is still
not well developed. In particular, in \cite{bmr1} a predual $H$ of the space
$W$ is constructed but the consequences have not been explored.

T3. Incidentally we note that the duality theory for the interpolation spaces
introduced in \cite{ckmr} has not been studied.

T4. Compactness is a natural issue that has not been considered so far in
abstract theory of commutators. For example, it is an important known result
that commutators of CZO and functions in $VMO$ generate compact operators on
$L^{p}$ (cf. \cite{uch})$.$ We believe that the framework proposed in this
paper could be useful to formulate the corresponding abstract result. In
particular, one can define an appropriate analog of $VMO.$..

T5. In connection with T3 and T4 it would be of interest to study compactness
(weak compactness) in the abstract setting of \cite{joma} using the ideas in
this paper.

\end{document}